\documentclass[12pt]{amsart}
\usepackage{latexsym}
\usepackage{amsmath}
\usepackage{amssymb}
\usepackage{amscd}
\usepackage{amsthm}
\usepackage{xypic}
\xyoption{curve}
\usepackage{ifthen} 
\usepackage{hyperref} 
\usepackage{graphicx}
\theoremstyle{plain}
 \newtheorem{theorem}{Theorem}[section]
 \newtheorem{proposition}[theorem]{Proposition}
 \newtheorem{lemma}[theorem]{Lemma}
 \newtheorem{corollary}[theorem]{Corollary}
 
\theoremstyle{definition}
 \newtheorem{definition}[theorem]{Definition}
 \newtheorem*{examples}{Examples}
 
\theoremstyle{remark}
 \newtheorem{remark}[theorem]{Remark}

\makeindex
\begin{document}
\title[Stable cohomology]{Isoclinism and stable cohomology of wreath products}

\author[Bogomolov]{Fedor Bogomolov$^1$}
\address{F. Bogomolov, Courant Institute of Mathematical Sciences\\
251 Mercer St.\\
New York, NY 10012, U.S.A., \emph{and}}
\address{Laboratory of Algebraic Geometry, GU-HSE\\
7 Vavilova Str.\\
Moscow, Russia, 117312}
\email{bogomolo@courant.nyu.edu}
\author[B\"ohning]{Christian B\"ohning$^2$}
\address{Christian B\"ohning, Fachbereich Mathematik der Universit\"at Hamburg\\
Bundesstra\ss e 55\\
20146 Hamburg, Germany}
\email{christian.boehning@math.uni-hamburg.de}

\thanks{$^1$ Supported by NSF grant DMS-1001662 and by AG Laboratory GU- HSE grant RF government ag. 11 11.G34.31.0023}
\thanks{$^2$ Supported by the German Research Foundation (Deutsche Forschungsgemeinschaft) through Heisenberg-Stipendium BO 3699/1-1}

\newcommand{\PP}{\mathbb{P}} 
\newcommand{\QQ}{\mathbb{Q}} 
\newcommand{\ZZ}{\mathbb{Z}} 
\newcommand{\CC}{\mathbb{C}} 
\newcommand{\rmprec}{\wp}
\newcommand{\rmconst}{\mathrm{const}}
\newcommand{\xycenter}[1]{\begin{center}\mbox{\xymatrix{#1}}\end{center}} 
\newboolean{xlabels} 
\newcommand{\xlabel}[1]{ 
                        \label{#1} 
                        \ifthenelse{\boolean{xlabels}} 
                                   {\marginpar[\hfill{\tiny #1}]{{\tiny #1}}} 
                                   {} 
                       } 
\setboolean{xlabels}{false} 

\

\begin{abstract}
Using the notion of isoclinism introduced by P. Hall for finite $p$-groups, we show that many important classes of finite $p$-groups have stable cohomology detected by abelian subgroups, see Theorem \ref{tMain}. Moreover, we show that the stable cohomology of the $n$-fold wreath product $G_n = \ZZ/p \wr \dots \wr \ZZ/p$ of cyclic groups $\ZZ/p$ is detected by elementary abelian $p$-subgroups and we describe the resulting cohomology algebra explicitly. Some applications to the computation of unramified and stable cohomology of finite groups of Lie type are given. 
\end{abstract}

\maketitle

\section{Introduction}\xlabel{sIntroduction}

Recall that for a prime $p$ and finite group $G$, the stable cohomology $H^*_{\mathrm{s}} (G, \: \ZZ /p )$ is defined as the quotient $H^*(G, \: \ZZ /p ) / N_{G, p}$ where, for some generically free $G$-representation $V$ with open part $V^L$ with free $G$-action, $N_{G, p}$ is the kernel of the map
\[
H^* (G, \: \ZZ /p )  \to  \varinjlim_U H^i (U/G , \: \ZZ /p\ZZ )\, ,
\]
the direct limit running over all nonempty Zariski open subsets $U\subset V^L$. In fact, $N_{G, p}$ is independent of the choice of $V$. In $H^*_{\mathrm{s}} (G; \: \ZZ/p)$ we have the subring of unramified elements $H^*_{\mathrm{nr}} (G, \: \ZZ/p)$; these play a vital role in the study of birational properties of generically free linear quotients $V/G$ and varieties $X$ in general, see \cite{Bogo93} for definitions and background. 

\

The object of this paper is to make $H^*_{\mathrm{s}}(G, \: \ZZ/p)$ (and $H^*_{\mathrm{nr}} (G, \: \ZZ/p)$) amenable to effective computation for rather important and large classes of groups. The development of the theory formally parallels that in the ordinary cohomology of finite groups: iterated wreath products of finite cyclic $p$-groups play an important part because they occur as building blocks of Sylow subgroups of a variety of classes of finite groups, in particular many finite groups of Lie type. In ordinary group cohomology, systematic studies along these lines were made in the famous papers by Quillen \cite{Quill71a}, \cite{Quill71b}, \cite{Quill71c}. Our treatment uses the notion of isoclinism of finite $p$-groups introduced by Hall in the paper \cite{Hall}. It turns out that generically free linear quotients by isoclinic groups are stably birational, see Theorem \ref{tComparisonIsoclinicToric}; this answers a question raised in \cite{HKK} (Question 1.11) in the affirmative. A partial result in this direction has been obtained previously in \cite{Mor} where it is proven that isoclinic groups have isomorphic unramified cohomology in degree $2$. Moreover, the stable cohomology of isoclinic groups share some important properties, see Proposition \ref{ptoroidal}, namely if $G_1$ and $G_2$ are isoclinic, then the stable cohomology of $G_1$ is detected by abelian subgroups if and only if the same is true for $G_2$. 
\

Here is a further outline of the contents of this paper: in Sections \ref{storoidal} and \ref{sIsoclinism} we prove the afore-mentioned results for isoclinic groups and, as an ingredient of the proof, we show that the notions of isoclinism and being toroidally related coincide for group extensions.\\
Section \ref{sWreathProducts} then contains the application of the results on isoclinic groups to the computation of the stable cohomology of iterated wreath products of groups $\ZZ/p$. The main results are Theorem \ref{tMain} and its Corollary \ref{cWreathAb} saying that for the stable cohomology of these groups we have detection by abelian subgroups. Theorem \ref{tWreathElAb} shows that for $G_n = \ZZ/p \wr \dots \wr \ZZ/p$ ($n$-factors) we even get detection by elementary abelian $p$-subgroups. This enables us to determine the structure of the stable cohomology algebra of $G_n$ in Theorem \ref{tFinerStrWreath} completely. In Section \ref{sLie} we give some applications of the preceding results to the computation of the unramified and stable cohomology of some finite groups of Lie type. A more intensive treatment will be given elsewhere. In particular, we recover (and extend) results from \cite{BPT} by this method.

\section{Toroidally related extensions}\xlabel{storoidal}
\begin{definition}\xlabel{dtoroidal}
Let $G$ be a finite group and let $A_1, \: A_2$ be finite abelian groups. If $e_1\, :\, 1 \to A_1 \to G_1 \to G \to 1$ and $e_2\, :\, 1\to A_2 \to G_2 \to G \to 1$ are two central extensions of $G$, we call them (resp. $G_1$ and $G_2$) \emph{toroidally related} if there is an algebraic torus $T\simeq (\CC^{\ast })^r$ together with embeddings $i_k\, :\, A_k \hookrightarrow T$, $k=1, \: 2$, such that 
the images of $e_1 \in H^2 (G, A_1)$ and $e_2 \in H^2 (G, A_2)$ in the cohomology $H^2 (G, \: T)$ coincide.
\end{definition}

\begin{examples}
(1) If $G$ is abelian, then one knows for the group homology $H_2 (G, \: \ZZ ) = \Lambda^2 G$, and the universal coefficient sequence for an arbitrary $G$-module $M$ reads
\[
0 \to \mathrm{Ext}^1 (G, \: M) \to H^2 (G, \: M) \stackrel{c}{\to} \mathrm{Hom} (\Lambda^2 G, \: M) \to 0 \, .
\]
For $M= \mathbb{Q}/\mathbb{Z}$, we have $\mathrm{Ext}^1 (G, \: \mathbb{Q}/\mathbb{Z})=0$ as $\mathbb{Q}/\mathbb{Z}$ is an injective $\ZZ$-module; hence an isomorphism
\[
c\, :\, H^2 (G, \: \mathbb{Q}/\mathbb{Z}) \stackrel{\sim }{\to } \mathrm{Hom} (\Lambda^2 G, \: \mathbb{Q}/\mathbb{Z})\, .
\]
More concretely, for any $M=A$ with trivial action, the map $c\, :\, H^2 (G, \: A) \to \mathrm{Hom} (\Lambda^2 G, \: A)$ can be described as follows: to a central extension 
\[
1 \to A \stackrel{i}{\to } \tilde{G} \stackrel{ \pi }{\to } G \to 1
\]
one associates the $A$-valued skew-form on $G$ given by the commutator: 
\[
i (c (g, \: h ) ) = \tilde{g}\tilde{h}\tilde{g}^{-1} \tilde{h}^{-1}\, , \; \mathrm{where}\; \pi (\tilde{g}) = g, \: \pi (\tilde{h}) = h \, .
\]
The kernel $\mathrm{Ext}^1 (G, \: A)$ of $c$ in this more general set-up can be identified with the abelian extensions $\mathcal{E}_{\mathrm{ab}} (G, \: A)$ of $G$ by $A$. From the short exact sequence of coefficients
\[
0 \to \mathbb{Q}/\mathbb{Z} \to \mathbb{C}^{\ast} \to \mathbb{C}^{\ast } /(\mathbb{Q} / \mathbb{Z})Ê\to 1
\]
we get $H^2 (G, \: \mathbb{Q}/\mathbb{Z}) \simeq H^2 (G, \: \mathbb{C}^{\ast })$ because $\mathbb{C}^{\ast }/(\mathbb{Q} / \mathbb{Z})$ is an infinite-dimensional vector space over $\mathbb{Q}$ (with ``vector addition" = multiplication and ``scalar multiplication"=exponentiation; this uses the algebraic closedness of $\CC$), whence $H^i (G, \: \mathbb{C}^{\ast } /(\mathbb{Q} / \mathbb{Z})) =0$ for $i>0$. If $i\, :\, A \hookrightarrow (\CC^{\ast })^r$ is an embedding, the diagram
\[
\begin{CD}
H^2 (G, \: A ) @>{c}>> \mathrm{Hom} (\Lambda^2 G, \: A)\\
@V{i^{\ast}}VV             @V{i\circ (-)}VV\\
H^2 (G, \: (\CC^{\ast })^r ) @>{c}>> \mathrm{Hom} (\Lambda^2 G, \: (\CC^{\ast })^r) \, .
\end{CD}
\]
Thus, for fixed abelian $G$ and $A$, two central extensions $e_1$ and $e_2$ of $G$ by $A$ are toroidally related if and only if they give the same skew form in $\mathrm{Hom} (\Lambda^2 G, \: A)$, or equivalently, there difference $e_1-e_2\in H^2 (G, \: A)$ represents the class of an abelian extension. If $A$ and $G$ are elementary abelian $p$-groups, this is the same as saying that $e_1-e_2$ lie in the subspace of $H^2 (G, \: A)$ spanned by the Bocksteins.

\

(2) In general, when $G$ is not necessarily abelian, we still have a universal coefficient sequence
\[
0 \to \mathrm{Ext}^1_{\mathbb{Z}} (H_1 (G, \: \ZZ ), \: A) \to H^2 (G, \: A ) \to \mathrm{Hom} (H_2 (G, \: \ZZ ) , \: A) \to 0
\]
\end{examples}
and, for an embedding $i\, :\,  A\hookrightarrow (\CC^{\ast })^r$ a diagram
\[
\begin{CD}
H^2 (G, \: A ) @>{c}>> \mathrm{Hom} (H_2 (G, \: \ZZ ), \: A)\\
@V{i^{\ast}}VV             @V{i\circ (-)}VV\\
H^2 (G, \: (\CC^{\ast })^r ) @>{c}>> \mathrm{Hom} (H_2 (G, \: \ZZ ), \: (\CC^{\ast })^r) \, 
\end{CD}
\]
where the bottom horizontal arrow is again an isomorphism. Thus extensions $e_1$ and $e_2$ are toroidally related in this case if and only if their difference $e_1-e_2$ is in the subspace $\mathrm{Ext}^1_{\mathbb{Z}} (H_1 (G, \: \ZZ ), \: A)$. Now $H_1 (G, \: \ZZ ) = G^{\mathrm{ab}}$ and the map $\mathrm{Ext}^1_{\mathbb{Z}} (G^{\mathrm{ab}}, \: A) \to H^2 (G, \: A ) $ factors 
\[
\mathrm{Ext}^1_{\mathbb{Z}} (G^{\mathrm{ab}}, \: A) \to H^2 (G^{\mathrm{ab}}, \: A) \stackrel{p^{\ast }}{\to } H^2 (G, \: A )
\]
where $p\, :\, G\to G^{\mathrm{ab}}$ is the projection and the arrow $\mathrm{Ext}^1_{\mathbb{Z}} (G^{\mathrm{ab}}, \: A) \to H^2 (G^{\mathrm{ab}}, \: A)$ is the one in the universal coefficient sequence
\[
0 \to \mathrm{Ext}^1 (G^{\mathrm{ab}}, \: A) \to H^2 (G^{\mathrm{ab}}, \: A) \stackrel{c}{\to} \mathrm{Hom} (\Lambda^2 G^{\mathrm{ab}}, \: A) \to 0 \, .
\]
This means that if $e_1-e_2$ are toroidally related, then their difference corresponds to an extension $\tilde{G}$ of $G$ by $A$ induced from an abelian extension $\tilde{G}^{\mathrm{ab}}$ of $G^{\mathrm{ab}}$ by $A$:
\[
\begin{CD}
1 @>>> A @>>> \tilde{G} @>>> G @>>> 1\\
@.   @|    @VVV   @VVV  @.\\
1 @>>> A @>>> \tilde{G}^{\mathrm{ab}} @>>> G^{\mathrm{ab}} @>>> 1\, .
\end{CD}
\]

\

The next Proposition shows that if $G_1$ and $G_2$ are toroidally related, then their stable cohomologies have important properties in common.

\begin{proposition}\xlabel{ptoroidal}
With the set-up of Definition \ref{dtoroidal} we have
\begin{itemize}
\item[(1)]
$H^*_{\mathrm{nr}} (G_1, \: \ZZ /p) \simeq H^*_{\mathrm{nr}} (G_2, \: \ZZ /p )$;
\item[(2)]
if $H^*_{\mathrm{s}} (G_1, \: \ZZ /p )$ is detected by abelian subgroups, so is $H^*_{\mathrm{s}} (G_2, \: \ZZ /p )$.
\end{itemize}
\end{proposition}

\begin{proof}
(1) Let $G_c$ be the central extension of $G$ by $T=(\CC^{\ast })^r$ determined by $G_1$ and $G_2$. Then $G_c$ is an extension $1\to G_i \to G_c \to T/A_i\to 1$. Then a generically free representation $V$ for $G_c$ gives a generically free representation for both $G_i$, and $V/G_i \to V/G_c$ is a torus principal bundle, hence locally trivial. In particular, $V/G_i$, $i=1, \: 2$, are stably birationally isomorphic and have the same unramified cohomology.

\

(2) Look at the fiber product of $G_1$ and $G_2$ over $G$:
\[
\begin{CD}
G_{12} = G_1\times_G G_2 @>{\pi_1}>> G_1 \\
@V{\pi_2}VV    @VVV \\
G_2 @>>> G\, .
\end{CD}
\]
Then $G_{12}$ is a toroidally trivial extension of both $G_1$ and $G_2$. We subdivide the proof into two auxiliary steps:
\begin{itemize}
\item[(a)]
If the stable cohomology of $G_1$ is detected by abelian subgroups, the same holds for the stable cohomology of $G_{12}$.
\item[(b)]
$H^{\ast}_{\mathrm{s}}(G_2, \: \ZZ/p)$ injects into $H^{\ast }_{\mathrm{s}} (G_{12}, \: \ZZ/p)$.
\end{itemize}
Given (a) and (b), one may conclude as follows: a nontrivial element $\alpha\in H^{\ast}_{\mathrm{s}}(G_2, \: \ZZ/p)$ is still nonzero in  $H^{\ast }_{\mathrm{s}} (G_{12}, \: \ZZ/p)$, hence is nonzero on some abelian subgroup $A\subset G_{12}$ by (a). Hence $\alpha$ will be nontrivial on $\pi_2 (A) \subset G_2$.\\
To prove (a) note that $G_{12}$, being a toroidally trivial extension of $G_1$, is induced from an abelian extension $\tilde{G}_1^{\mathrm{ab}}$ of $G_1^{\mathrm{ab}}$: 
\[
\begin{CD}
1 @>>> K @>>> G_{12} @>>> G_1 @>>> 1\\
@.   @|    @VVV   @VVV  @.\\
1 @>>> K @>>> \tilde{G}_1^{\mathrm{ab}} @>>> G_1^{\mathrm{ab}} @>>> 1\, .
\end{CD}
\]
Since $\tilde{G}_1^{\mathrm{ab}}$ is a finitely generated abelian group, there is a surjection $\hat{\ZZÊ}^r \twoheadrightarrow \tilde{G}_1^{\mathrm{ab}}$, hence by the universal property of fiber products we get a map $\hat{\ZZ }^r \times G_1 \to G_{12}$. We claim that the stable cohomology of $G_{12}$ injects into the stable cohomology of the profinite group $\hat{\ZZ } \times G_1$; this is so because $G_{12}$ has a generically free representation $V = \bigoplus_{\chi \in K^{\ast }} \mathrm{Ind}_{K}^{G_{12}} \CC_{\chi }$ such that $V/G_{12}$ is a torus bundle over an open subvariety of the product of projective spaces $\mathbb{P} = \prod_{\chi \in K^{\ast}} \mathbb{P} (\mathrm{Ind}_{K}^{G_{12}} \CC_{\chi })$. It is trivial over an open subvariety $U$ in $\mathbb{P}$, and the profinite completion of the corresponding fundamental group of $U \times (\CC^{\ast })^r$, $\hat{\ZZ}^r\times G_1'$, maps to $\hat{\ZZ } \times G_1$ and gives a partial stabilization for the cohomology of $G_{12}$. Hence the cohomology of $G_{12}$ is stabilized by $\hat{\ZZ } \times G_1''$ where $G_1''$ is a group stabilizing the cohomology of $G_1$ (note that we have a K\"unneth formula in this setting as the stable cohomology of $G_1$ is detected by abelian subgroups). In other words the map
\[
H^* (G_{12}, \: \ZZ/p) \to H^{\astÊ}_{\mathrm{s}} (\hat{\ZZ } \times G_1, \: \ZZ/p)
\]
gives stabilization for $H^* (G_{12}, \: \ZZ/p)$. As the stable cohomology of $\hat{\ZZ } \times G_1$ is detected by abelian subgroups, claim (a) follows.

\

Claim (b) follows from the fact that by the above reasoning also
\[
H^* (G_{12}, \: \ZZ/p) \to H^{\astÊ}_{\mathrm{s}} (\hat{\ZZ } \times G_2, \: \ZZ/p)
\]
gives stabilization, but $H^{\ast}_{\mathrm{s}} (G_2, \: \ZZ/p)$ injects into $H^{\astÊ}_{\mathrm{s}} (\hat{\ZZ } \times G_2, \: \ZZ/p)$.
\end{proof}

\section{Comparison to the notion of isoclinism}\xlabel{sIsoclinism}

In \cite{Hall}, P.Hall introduced the notion of \emph{isoclinism} of finite groups which morally speaking means that the two groups have the same commutator function.

\begin{definition}\xlabel{dIsoclinism}
Two finite groups $G_1$ and $G_2$ with centers $Z_1$ and $Z_2$ are said to be \emph{isoclinic} if there are isomorphisms
\[
i\, :\, G_1/Z_1 \to G_2/Z_2, \quad j\, :\, [G_1, \: G_1] \to [G_2, \: G_2]
\]
such that
\[
\begin{CD}
G_1/Z_1\times G_1/Z_1 @>{(i, \: i)}>> G_2/Z_2\times G_2/Z_2 \\
@V{[ \cdot , \cdot ]}VV    @V{[\cdot , \cdot ]}VV \\
[G_1, \: G_1] @>{j}>> [G_2, \: G_2]\, 
\end{CD}
\]
commutes.
\end{definition}

As was remarked above, if $G_1$ and $G_2$ are toroidally related extensions of the group $G$, then $G_{12}$ is a toroidally trivial extension of both $G_1$ and $G_2$; which means that it is an extension induced from an abelian extension of $G_1^{\mathrm{ab}}$ resp. $G_2^{\mathrm{ab}}$. Hence $G_1$ and $G_2$ have the same commutator function (that of $G_{12}$) and are isoclinic.

\

Suppose conversely that $G_1$ and $G_2$ are isoclinic. Then naturally $G= G_1/Z_1 \simeq G_2/Z_2$ and we want to show that 
\begin{gather*}
e_1\, :\, 1 \to Z_1 \to G_1 \to G \to 1 \, , \\
e_2\,:\, 1\to Z_2 \to G_2 \to G \to 1
\end{gather*}
are toroidally related. We have to show that $e_1$ and $e_2$ map to the same element under $\iota_1\circ \alpha_i$ where $\alpha_i$ is the map in the sequence
\[
\begin{CD}
0 \to \mathrm{Ext}^1_{\mathbb{Z}} (H_1 (G, \: \ZZ ), \: Z_i) \to H^2 (G, \: Z_i ) @>{\alpha_i}>> \mathrm{Hom} (H_2 (G, \: \ZZ ) , \: Z_i) \to 0\, .
\end{CD}
\]
and $\iota_i$ is the map
\[
\iota_i \, :\, \mathrm{Hom} (H_2 (G, \: \ZZ ) , \: Z_i) \to  \mathrm{Hom} (H_2 (G, \: \ZZ ) , \: (\CC^{\ast})^r)
\]
induced by the embeddings $Z_i \hookrightarrow ( \CC^{\ast })^r$. 
To do this we use an interpretation of the Schur multiplier $H_2 (G, \: \ZZ )$ from \cite{Karp}, section 2.6, in terms of commutator relations. Let $\langle G, \: G\rangle$ be the free group generated by all pairs $\langle x, \: y\rangle$ with $x, \: y\in G$ together with its natural map $c \, :\, \langle G, \: G \rangle \to [G, \: G]$ with $c(\langle x, \: y\rangle ) = [x, \: y] = xyx^{-1}y^{-1}$. The kernel of $p$, denoted by $C(G)$, consists of relations among commutators in $G$. Moreover, there are the following universal commutator relations valid in any group $G$:
\begin{gather*}
\langle x, \: x \rangle , \: \langle x, \: y\rangle \langle y, \: x\rangle , \: \langle y, \: z \rangle^x \langle x, \: z\rangle \langle xy, \: z\rangle^{-1} \, , \\
\langle y, \: z\rangle^x \langle y, \: z\rangle^{-1} \langle x, \: [y, \: z] \rangle^{-1}
\end{gather*}
where for $x,\: y, \: z\in G, \: \langle y, \: z\rangle^x = \langle xyx^{-1}, \: xzx^{-1}\rangle$. The smallest normal subgroup in $C(G)$ containing all these universal relations is denoted by $B(G)$. Let $H(G)=C(G)/B(G)$ be the quotient. The universal relations are the ones that hold in a free group. We have now, by Theorem 2.6.6 of \cite{Karp}, that naturally $H(G) \simeq H_2(G, \: \ZZ )$, the Schur multiplier, for a finite group $G$. In fact this is a consequence of Hopf's formula for $H_1 (G, \: \ZZ )$ which says that if $F \twoheadrightarrow G$ is a free presentation of $G$ with subgroup of relations $R$, then $H_2(G, \: \ZZ ) \simeq (R \cap [F, \: F] )/ [F, \: R]$.

Now we want to reinterpret the maps 
\[
\iota_i \circ \alpha_i \, :\, H^2 (G, \: Z_i) \to \mathrm{Hom} (H(G), \: Z_i) \hookrightarrow \mathrm{Hom} (H(G), \: ( \CC^{\ast })^r)
\] 
(see also the proof of Theorem 2.6.6 in \cite{Karp}): if $A$ is one of $Z_1$ or $Z_2$, then to a central extension $e\in H^2 (G, \: A)$ given by
\[
1\to A \to \tilde{G} \to G \to 1
\]
we first associate the homomorphism $\langle G, \: G \rangle \to \tilde{G}$ which maps $\langle x, \: y\rangle$ to $[\bar{x}, \: \bar{y}]$. This homomorphism maps $C(G)$ into $A$ (in fact onto $A\cap [\tilde{G}, \: \tilde{G}]$), and it maps $B(G)$ to $\{ 1 \}$, hence associated to $e$ one gets a homomorphism 
\[
\psi_e \, :\, H(G) \to A \cap [\tilde{G}, \: \tilde{G}] \subset A \, .
\] 
Then $\alpha (e) = \psi_e$. It is obvious that if $e_1$ and $e_2$ have the same commutator function, then $\psi_{e_1} = \psi_{e_2}$, viewed as maps into $(\CC^{\ast})^r$. Hence we have proven

\begin{theorem}\xlabel{tComparisonIsoclinicToric}
The notions of being toroidally related and isoclinic coincide. In particular, if $G_1$ and $G_2$ are isoclinic, then generically free linear quotients for $G_1$ and $G_2$ are stably equivalent. 
\end{theorem}

This answers Question (1.11) of \cite{HKK} in the affirmative; the partial result that $G_1$ and $G_2$ have isomorphic second unramified cohomology groups has been proven in \cite{Mor}.

\

We give an additional elementary argument for the implication ``$G_1$ and $G_2$ isoclinic" $\implies$ ``the extensions $e_1$ and $e_2$ are toroidally related". Suppose that $G_1$ and $G_2$ are isoclinic, and consider the fiber product over $G$ as above:
\[
\begin{CD}
G_{12} = G_1\times_G G_2 @>{\pi_1}>> G_1 \\
@V{\pi_2}VV    @VVV \\
G_2 @>>> G\, .
\end{CD}
\]
We want to show that then the extensions $1\to Z_2\to G_{12} \to G_1 \to 1$ and $1\to Z_1\to G_{12} \to G_2 \to 1$ are toroidally trivial. Note that the preimages of $[G_1, \: G_1] \subset G_1$ and $[G_2, \: G_2]\subset G_2$ coincide with $[G_{12}, \: G_{12}]$ and this group equals 
\[
[G_{12}, \: G_{12}] = \langle \left( [g_1, \: h_1], \: j ([g_1, \: h_1]) \right) \rangle \, \quad g_1, \: h_1 \in G_1
\]
which is a ``diagonal subgroup" of $G_{12}$ which intersects both $Z_1 \simeq \{ (z_1, \: 1)\, | \, z_1\in Z_1 \} \subset G_{12}$ and $Z_2 \simeq \{ (1, \: z_2)\, | \, z_2\in Z_2 \}$ trivially. In other words, $Z_1$ maps isomorphically to $G_2/[G_2, \: G_2]=G_2^{\mathrm{ab}}$ and $Z_2$ isomorphically to $G_1/[G_1, \: G_1]=G_1^{\mathrm{ab}}$. Therefore, for example, the extension $G_{12}$ of $G_2$ by $Z_1$ is induced by an abelian extension
\[
1 \to Z_1 \to G_{12}/[G_{12}, \: G_{12}]  \to G_2/[G_2, \: G_2] \to 1\, .
\]
The same holds for the extension $G_{12}$ of $G_1$. Hence these two are toroidally trivial, hence $e_1$ and $e_2$ are toroidally related.

\

The following Remark is sometimes useful and summarizes some compatibilities of isoclinism with passing to subgroups or quotients. It can be found already in \cite{Hall}.

\begin{remark}\xlabel{rIsoclinismProperties}
An isoclinism between $G_1$ and $G_2$ sets up a bijective correspondence between subgroups of $G_1$ containing $Z_1$ and subgroups of $G_2$ containing $Z_2$, and corresponding subgroups are isoclinic.\\
In particular, a centralizer in the group $G_2$ is isoclinic to a centralizer in $G_1$.\\
Moreover, an isoclinism also gives a bijective correspondence between quotient groups $G_1/K_1$ and $G_2/K_2$, where $K_1\subset [G_1, \: G_1]$ and $K_2\subset [G_2, \: G_2]$, and corresponding quotient groups are isoclinic. 
\end{remark}

\section{Stable cohomology of wreath products}\xlabel{sWreathProducts}
Now we want to use Proposition \ref{ptoroidal} to compute the stable cohomology of the iterated wreath product $G_n = \ZZ /p \wr \ZZ /p \wr \dots \wr \ZZ /p$ of groups $\ZZ /p$ (there are $n$-factors $\ZZ /p$). We first define a class of groups which will turn out to be stable under taking iterated centralizers, provided one identifies groups which are isoclinic/toroidally related, and contains $G_n$.

\begin{definition}\xlabel{dClassCp}
Let $\mathcal{D}$ be a class of groups with the property that (1) every group in $\mathcal{D}$ has stably rational generically free linear quotients, and (2) the centralizer of any element in a group in $\mathcal{D}$ again belongs to $\mathcal{D}$.\\
Then we define a group to belong to the class $\mathcal{C}_p(\mathcal{D})$ if it can be reached starting from a group in $\mathcal{D}$ by a finite number of the following operations, which successively enlarge the set of groups in $\mathcal{C}_p(\mathcal{D})$ already constructed:
\begin{itemize}
\item[(a)]
taking a wreath product with a group $\ZZ/p$, i.e. passing from $H$ to $H\wr \ZZ/p$,
\item[(b)]
taking a finite direct product,
\item[(c)]
passing from $H$ to an isoclinic group $H'$.
\end{itemize}
\end{definition}

It follows that all the groups in $\mathcal{C}_p (\mathcal{D})$  have trivial higher unramified cohomology, in fact they all have stably rational quotients. Moreover, clearly $G_n$ belongs to $\mathcal{C}_p$ if we take $\mathcal{D}$ to consist only of the group $\ZZ/p$. 

\

The key result will be Proposition \ref{pInductionWreath} below. We precede it with a Lemma on isoclinism types of centralizers in a wreath product $H=H'\wr\ZZ/p$ which will be used in the proof of Proposition \ref{pInductionWreath}.

\begin{lemma}\xlabel{lIsoclinismTypeCentralizersWreath}
Let $H'$ be a finite group, and let $H=H'\wr \ZZ/p$ be its wreath product with $\ZZ/p$. Let $x\in H$ be some element, and $Z_H(x)$ be its centralizer. Then one of the following is true:
\begin{itemize}
\item[(a)]
The element $x$ is contained in $(H')^{p}\subset H$, $Z_H(x)$ does \emph{not} surject onto the quotient $\ZZ/p$ under the natural projection $H\twoheadrightarrow \ZZ/p$, and the centralizer $Z_H(x)$ is a product 
\[
Z_{H'} (x_1) \times \dots \times Z_{H'} (x_{p})
\]
of the centralizers $Z_{H'}(x_i)$ of the components $x_i$ of $x$ with respect to the product $H'\times ...\times H'$ ($p$-factors) in $H$.
\item[(b)]
The cyclic subgroup $\langle x\rangle$ generated by $x$ in $H$ surjects onto the quotient $\ZZ/p$ under the natural projection $H\twoheadrightarrow \ZZ/p$, and $Z_H (x)$ is \emph{isoclinic} to 
\[
Z_{H'} (x^{p}) \times \ZZ/p\, .
\]
\item[(c)]
The element $x$ is contained in $(H')^p$ and $Z_{H}(x)$ surjects onto $\ZZ/p$. Then $x= (x', \dots , \: x') \in (H')^p$ and 
\[
Z_{H'}(x) \wr \ZZ/p\, .
\]
\end{itemize}
\end{lemma}

\begin{proof}
The cases enumerated in (a), (b), (c) obviously cover all the possibilities and are mutually exclusive. We deal with them one by one.

\

\textbf{Case (a)}
The centralizer $Z_{H} (x)$ is contained in $(H')^{p}$. Then $Z_{H} (x)$ is obviously the product of the centralizers of components.

\

\textbf{Case (b)}
Suppose that $\langle x \rangle$ surjects onto the quotient $\ZZ/p$ of $H=  H' \wr \ZZ /p$.  After conjugating by an element in $(H')^{p}$ we may assume $x = ((a, \: \mathrm{id}, \dots , \: \mathrm{id}), \: \sigma )$ for some $a\in H'$ and $\sigma\in \ZZ/p$ a generator. Indeed, if a priori $x = ((x_1, \dots , \: x_p), \: \sigma )$,  it is sufficient for this to solve the equations $c_1 x_1 c_2^{-1} =a, \dots , \: c_{p-1} x_{p-1} c_p^{-1} = \mathrm{id}, \: c_p x_p c_1^{-1}= \mathrm{id}$ in elements $c_i$ of $H'$, which is always possible successively, and conjugate by $c= ((c_1, \dots , \: c_p), \: \mathrm{id})$. We will assume therefore now that $x=((a, \: \mathrm{id}, \dots , \: \mathrm{id}), \: \sigma )$ with $\sigma$ a generator of $\ZZ/p$ for simplicity.\\
Clearly, $Z_{H} (x)$ is generated by $\langle x \rangle$ and those $y=((y_1, \dots , \: y_{p}), \: \mathrm{id})$ which commute with $x$. These elements $y$ will also commute with all powers of $x$, hence with $((a, \: a, \dots , \: a), \: \mathrm{id})$. That is, it is necessary that $a$ commutes with every $y_j$. But then the elements $y$ which commute with $x$ are precisely those such that $y_1 = \dots = y_{p}$ and all of them commute with $a$.

In this case it follows that $Z_H (x)$ is an extension
\[
0 \to Z_{H'} (a) \to Z_H (x ) \to \ZZ / p \to 0
\]
which will in general be nontrivial; however, we claim that $Z_H (x)$ is toroidally related to the product $Z_{H'} (a) \times \ZZ /p$. More precisely, $Z_H (x)$ is a central extension
\[
1 \to \langle a \rangle \to Z_H (x) \to Z_H (x) / \langle a \rangle \simeq Z_{H'} (a)/\langle a \rangle \times \ZZ /p \to 1
\]
of $Z_{H'} (a)/\langle a \rangle \times \ZZ /p$ by the cyclic group generated by $a$. We claim that this extension is toroidally related to the extension
\[
1 \to \langle a \rangle \to Z_{H'} (a) \times \ZZ /p \to Z_{H'} (a)/\langle a \rangle \times \ZZ /p \to 1 \, .
\]
Indeed, look at the extension $\widetilde{Z_H (x)}$ of $Z_{H'} (a)/\langle a \rangle \times \ZZ /p$ by $\mathbb{Q} /\mathbb{Z}$ induced by $Z_H (x)$:
\begin{center}
\xymatrix{
           &                                                     &               1\ar[d]                             &                                                                                &       \\
            &                                                     &               Z_H(x) \ar[d]        &                                                                                &       \\
1 \ar[r] & \mathbb{Q}/\mathbb{Z} \ar[r]\ar[dr]^{\cdot N} & \widetilde{Z_H (x)} \ar[r]\ar[d] & Z_{H'} (a)/\langle a \rangle \times \ZZ /p \ar[r] & 1\\
            &                                                     &              \mathbb{Q}/\mathbb{Z}\ar[d]        &                                                                                &       \\
                       &                                                     &               1                             &                                                                                &       
}
\end{center}
As $\mathbb{Q}/\mathbb{Z}$ is divisible, the element $a\in Z_{H'} (a) \subset Z_H (x)$ is a $p$-th power of a central element $A$ in $\widetilde{Z_H (x)}$. In particular, $\widetilde{Z_H (x)}$ contains also $Z_{H'} (a) \times \ZZ /p$; we can map a generator of $\ZZ /p$ to a lift of it in $Z_H (x)$ multiplied by $A^{-1}$.

\

\textbf{Case (c)}
We have $x\in (H')^{p}$, but there is an element $g\in Z_H (x)$ such that the subgroup generated by it surjects onto $\ZZ/p$. The element $g$ is then conjugate to $((a, \: \mathrm{id}, \dots , \: \mathrm{id}), \: \sigma )$ for some $a\in H'$ and $\sigma\in \ZZ/p$ a generator. We can assume then that $g=((a, \: \mathrm{id}, \dots , \: \mathrm{id}), \: \sigma )$. Now $g^{p} = ((a, \: a, \dots , \: a), \: \mathrm{id})\in Z_{(H')^{p}} (x)$, and as this is a product of the centralizers of the components of $x$, also $b=((a, \: \mathrm{id}, \dots , \: \mathrm{id}), \: \mathrm{id})\in Z_{H} (x)$. So $b^{-1} \cdot g = ( (\mathrm{id}, \dots , \: \mathrm{id} ), \: \mathrm{\sigma})$ is also always in $Z_{H} (x)$ in this case. Hence $x = (x', \dots , \: x')$ here, and $Z_{H} (x) = Z_{H'} (x')\wr \ZZ /p$.
\end{proof}

\begin{proposition}\xlabel{pInductionWreath}
Suppose $G$ is a group in $\mathcal{C}_p (\mathcal{D})$ and let $h\in G$ be some element. Then the centralizer $Z_G (h)$ is again a group in $\mathcal{C}_p(\mathcal{D})$.
\end{proposition}

\begin{proof}
We do induction over the number of steps it takes to reach a particular group in $\mathcal{C}_p (\mathcal{D})$ in a certain fixed succession of construction steps that give all the groups in $\mathcal{C}_p (\mathcal{D})$, starting from some group in $\mathcal{D}$ (i.e., we fix a particular ordered sequence of steps for the construction of all groups in $\mathcal{C}_p(\mathcal{D})$ from those in $\mathcal{D}$).

\

The beginning of the induction is trivial because the assertion of Proposition \ref{pInductionWreath} holds by assumption for all the groups in $\mathcal{D}$. Now suppose it holds for all groups in $\mathcal{C}_p (\mathcal{D})$ constructed up to a certain step $s$. We call this subset $\mathcal{ C}_p^{\le s} (\mathcal{D})$. Suppose then $H$ is a new group  constructed out of $\mathcal{C}_p^{\le s} (\mathcal{D})$ according to the rules in Definition \ref{dClassCp}. We have the following possibilities.
\begin{itemize}
\item[(1)] The group $H$ is a finite product
\[
H = H_1' \times \dots \times H'_N
\]
of groups $H_i'$ in $\mathcal{C}_p^{\le s} (\mathcal{D})$. Then the centralizer $Z_H (x)$ of an element $x\in H$ is the product of the centralizers of the components $x_i$ of $x$. Each of the $Z_{H_i}(x_i)$ belongs to $\mathcal{C}_p (\mathcal{D})$ by the induction hypothesis, hence so does the product as $\mathcal{C}_p (\mathcal{D})$ is closed under taking finite products by definition.
\item[(2)]
The group $H$ is isoclinic to a group $H'$ in $\mathcal{C}_p^{\le s}(\mathcal{D})$. By Remark \ref{rIsoclinismProperties}, centralizers of elements in $H$ are isoclinic to centralizers of elements in $H'$. The latter however belong to $\mathcal{C}_p (\mathcal{D})$ by induction. As $\mathcal{C}_p (\mathcal{D})$ is closed under passage to isoclinic groups, $H$ belongs to $\mathcal{C}_p (\mathcal{D})$, too, in this case.
\item[(3)]
The group $H$ is a wreath product 
\[
H = H' \wr \ZZ /p
\]
where $H'$ belongs to $\mathcal{C}_p^{\le s} (\mathcal{D})$. According to Lemma \ref{lIsoclinismTypeCentralizersWreath} above, we see, using the induction hypothesis and the definition of the class $\mathcal{C}_p (\mathcal{D})$, that $H$ also belongs to $\mathcal{C}_p(\mathcal{D})$. 
\end{itemize}

This concludes the proof.
\end{proof}

\begin{theorem}\xlabel{tMain}
The stable cohomology $H^*_{\mathrm{s}} (G, \: \ZZ /p)$ is detected by abelian subgroups for any group $G$ in $\mathcal{C}_p (\mathcal{D})$. 
\end{theorem}

\begin{proof}
Here we use Lemma 1.5 of \cite{B-B11} inductively. Then everything follows from Proposition \ref{pInductionWreath}, saying that $\mathcal{C}_p (\mathcal{D})$ is closed under taking centralizers, and induction over the cohomological degree: note that all groups in $\mathcal{C}_p (\mathcal{D})$ have trivial higher unramified cohomology and that the stable cohomology $H^1_{\mathrm{s}} (G, \: \ZZ/p)$ of any finite group $G$ is detected by abelian subgroups: this is so because any nontrivial character $\chi \, :\, G \to \ZZ/p$ is nontrivial on a cyclic subgroup in $G$.
\end{proof}

\begin{corollary}\xlabel{cWreathAb}
The stable cohomology $H^*_{\mathrm{s}} (G_n, \: \ZZ /p)$ is detected by abelian subgroups. 
\end{corollary}

\begin{proof}
Take $\mathcal{D} = \{ \ZZ/p \}$ and apply Theorem \ref{tMain}.
\end{proof}

The following result allows us to determine $H^*_{\mathrm{s}} (G_n, \: \ZZ /p)$ rather precisely. It follows from Corollary \ref{cWreathAb}, but requires some additional work.

\begin{theorem}\xlabel{tWreathElAb}
The stable cohomology $H^*_{\mathrm{s}} (G_n, \: \ZZ /p)$ is detected by elementary abelian subgroups.
\end{theorem}

\begin{proof}
It will be sufficient to prove:
\begin{quote}
Every nontrivial class $\alpha \in H^k_{\mathrm{s}} (G_n, \: \ZZ/p)$, $k>1$, is nontrivial on the subgroup $G_{n-1}^p\subset G_{n-1}\wr \ZZ/p = G_n$.
\end{quote}
Then the assertion of Theorem \ref{tWreathElAb} will follow by induction and the fact that $H^1_{\mathrm{s}} (G_n, \: \ZZ/p)$ is always detected by elementary abelian subgroups: in fact, every nontrivial character $\chi\, : \, G_n \to \ZZ/p$ is nontrivial on $G_{n-1}^p$ or else nontrivial on the quotient $\ZZ/p$.\\
By Theorem \ref{cWreathAb} a class $\alpha$ as above which is trivial on $G_{n-1}^p$ must be nontrivial on some abelian subgroup $A$ which surjects onto $\ZZ/p$ under the composite map $A\hookrightarrow G_n \twoheadrightarrow \ZZ/p$. Then $A$ is contained in a subgroup $B\wr \ZZ/p \subset G_n$ where $B^p \subset G_{n-1}^p$ is abelian: we can take for $B$ the image of $A\cap G_{n-1}^p$ in $G_{n-1}$ under any of the coordinate projections $G_{n-1}^p \to G_{n-1}$. Note that if $x = ((a, \: 1, \: 1, \dots , \: 1), \: \sigma )$ is an element in $A$ such that $\langle x \rangle$ surjects onto $\ZZ/p$ and if $b=((b_1, \dots , \: b_p ), \: 1) \in A \cap G_{n-1}^p$, then the equation $x b x^{-1}=b$ together with the fact that $a$ commutes with every $b_i$ implies that $b_1=\dots = b_p$, so all coordinate projections are the same.\\
Thus we get a nontrivial class in $H^k_{\mathrm{s}} (B \wr \ZZ/p, \: \ZZ/p)$ which with slight abuse of notation we denote again by $\alpha$. We have to recall some results about the structure of the cohomology of $B\wr \ZZ/p$: by Nakaoka's Theorem (see \cite{Evens}) one has an isomorphism
\[
H^*(B \wr \ZZ/p, \: \ZZ/p) \simeq H^* (\ZZ/p, \: H^*(B, \: \ZZ/p)^{\otimes p})
\]
where we consider the cohomology $H^* (B, \: \ZZ/p)^{\otimes p}$ as a $\ZZ/p[\ZZ/p]$-module (with nontrivial action). However, $H^* (B, \: \ZZ/p)^{\otimes p}$ is a direct sum of trivial $\ZZ/p[\ZZ/p]$-modules and free $\ZZ/p[\ZZ/p]$-modules (see \cite{A-M}, p. 117). The trivial modules are generated by norm elements $x\otimes \dots \otimes x \in H^*(B, \: \ZZ/p)^{\otimes p}$, $x\in H^*(B, \: \ZZ/p)$. The free modules do not contribute to the cohomology. Let $b_1, \: b_2, \dots$ be a basis for $H^*(B, \: \ZZ/p)$. Hence there is a natural splitting
\begin{gather*}
H^* (B \wr \ZZ/p, \: \ZZ/p) = H^0 (\ZZ/p, \: H^*(B, \: \ZZ/p)^{\otimes p}) \oplus \bigoplus_{k>0, \: i} H^k (\ZZ/p, \: T_{i})\\
\simeq H^* (B^p, \: \ZZ/p)^{\ZZ/p}\oplus \bigoplus_{k>0, \: i} H^k (\ZZ/p, \: T_{i})\, ,
\end{gather*}
the direct sum running over all the trivial modules $T_{i}$, generated by $b_i\otimes \dots \otimes b_i$, which occur.\\
Consider now a faithful toric representation $R_B$ for $B$ with open part $R_B^o$ where the $B$ action is free that stabilizes the cohomology of $B$. We construct the faithful $B\wr \ZZ/p$ representation $R_B^p \oplus \CC$ where $\ZZ/p$ acts via a $p$-th root of unity in $\CC$ and rotates the copies of $R_B$. This has an open toric free part $(R_B^o)^p \times \CC^{\ast}$ and the quotient $Q:=((R_B^o)^p \times \CC^{\ast})/ (B\wr \ZZ/p)$ has the structure of a torus fibration with fibre $(R_B^o)^p/B^p$ over $\CC^{\ast } \simeq (\CC^{\ast } )/ (\ZZ /p)$. The fundamental group $\pi_1 (Q)$ yields a partial stabilization for the cohomology of $B\wr \ZZ/p$ and is of the form $B^s \rtimes \ZZ$ where $B^s = \pi_1 (R_B^o/B)$ and stabilizes the cohomology of $B$. We consider the image of the cohomology of $B\wr \ZZ/p$ in the cohomology of $B^s \wr \ZZ/p$. That map can be factored
\[
\begin{CD}
H^* (B \wr \ZZ/p, \: \ZZ/p ) @>{f_1}>> H^* (B^s \wr \ZZ/p, \: \ZZ/p ) @>{f_2}>> H^* (B^s \rtimes \ZZ , \: \ZZ/p)
\end{CD}
\]
and the cohomologies of $B^s\wr \ZZ/p$ and also of $B^s\rtimes \ZZ$ can be described analogously to what was said above: first, clearly,
\[
H^* (B^s \wr \ZZ/p, \: \ZZ/p) \simeq H^* (\ZZ/p, \: H^*(B^s, \: \ZZ/p)^{\otimes p})
\]
and the description is entirely the same as before. Now for $B^s \rtimes \ZZ$ we also have, see \cite{Evens}, discussion on page 19 and proof of Theorem 5.3.1, that
\[
H^* (B^s \wr \ZZ , \: \ZZ/p ) \simeq H^* (\ZZ , \: H^* (B^s, \: \ZZ/p)^{\otimes p})
\]
where now we consider $H^* (B^s, \: \ZZ/p)^{\otimes p}$ as a $\ZZ/p [\ZZ]$ module via the quotient map $\ZZ \to \ZZ/p$. In particular, the free $\ZZ/p[\ZZ/p]$-submodules of $H^*(B^s, \: \ZZ/p)^{\otimes p}$ may contribute to the cohomology of $H^* (\ZZ , \: H^* (B^s, \: \ZZ/p)^{\otimes p})$ now (but those classes do not come from $H^* (B \wr \ZZ/p, \: \ZZ/p )$). The maps $f_1$ and $f_2$ have a natural description using the previous isomorphisms: the map $f_1$ is just induced by the map of coefficients $H^*(B, \: \ZZ/p)^{\otimes p} \to H^*(B^s, \: \ZZ/p)^{\otimes p}$ (which is a stabilization map for the cohomology of $B^p$), and $f_2$ is induced by the natural surjection of groups $\ZZ \twoheadrightarrow \ZZ/p$.\\
From this description we see that if a class $\alpha$ is in the subspace $H^* (B^p, \: \ZZ/p)^{\ZZ/p}$ and stable, then it is detected already on $B^p$. Moreover, the classes in $H^k (\ZZ/p, \: T_{i})$ can only be stable if $k=1$ and $T_{i}$ is generated by $\beta \otimes \dots \otimes \beta$ with $\beta$ stable (and part of the chosen basis for $H^* (B, \: \ZZ/p)$).  Let $\mathrm{deg} (\beta ) =: b$, and let us show that in fact all classes $\alpha =\tau \cup (\beta \otimes \dots \otimes \beta )$ with $b>0$ and $\tau$ some generator of $H^1 (\ZZ/p, \: \ZZ/p)$ are unstable in $H^* (B\wr \ZZ/p, \: \ZZ/p)$, which will prove Theorem \ref{tWreathElAb}. The degree of $\alpha$ is $1 + pb$. However, $\alpha$ is then induced from a class $\alpha'$ in $E \wr \ZZ/p$, where $E \simeq (\ZZ/p)^b$ is elementary abelian, via some surjection $B \twoheadrightarrow E$: in fact, we may assume $\beta$ is a monomial in $e_1, \dots , e_r$, the latter being some basis of $H^1 (B, \: \ZZ/p)$ and then the surjection is just a coordinate projection followed by reduction to $\ZZ/p$. If $\alpha$ were stable, then $\alpha'$ would be stable. This is however clearly not so if $b>0$ for in that case $E\wr \ZZ/p$ has a faithful representation of dimension $pb$ (let each standard copy of $\ZZ/p$ in $E \simeq \ZZ/p \times \dots \times \ZZ/p$ act on $\CC$ via a nontrivial character and let $\ZZ/p$ rotate those copies). Hence all classes in the cohomology of $E\wr \ZZ/p$ of degrees $> pb$ are killed under stabilization for dimension reasons. 
\end{proof}

We have already proven a lot more, but let us record the easy

\begin{corollary}\xlabel{cWreathElAb}
The stable cohomology $H^*_{\mathrm{s}} (G_n, \: \ZZ/p)$ is detected by the two subgroups $G_{n-1}^p$ and $G_{n-1}\times \ZZ/p$.
\end{corollary}

\begin{proof}
It is known (\cite{Mui}, p.349) that every maximal elementary abelian $p$-subgroup of $G_n$ is contained in $G_{n-1}^p$ or $G_{n-1}\times \ZZ/p$. 
\end{proof}

We can now also say very precisely how the cohomology ring $H^*_{\mathrm{s}}(G_n \: \ZZ/p)$ is structured. 

\begin{theorem}\xlabel{tFinerStrWreath}
The stable cohomology $H^*_{\mathrm{s}} (G_n, \ZZ/p)$ is determined inductively as follows: there is an isomorphism
\[
i\, :\, H^*_{\mathrm{s}} (G_n, \: \ZZ/p) \simeq H^* (G_{n-1}^p, \: \ZZ/p )^{\ZZ/p} \oplus H^1(\ZZ/p, \: \ZZ/p)
\]
where $i = (i_1, \: i_2)$ and $i_1$ is the restriction map $\mathrm{res}^{G_n}_{G_{n-1}^p}$ and $i_2$ is the restriction map to the subgroup $\ZZ/p$ in $G_n = G_{n-1} \wr \ZZ/p$.
\end{theorem}

\begin{proof}
This follows rather easily from the proof of Theorem \ref{tWreathElAb}, but let us give another proof using Steenrod operations and the Bloch-Kato conjecture to show how everything falls into place.

\

By Steenrod's description of the cohomology of wreath products (see e.g. \cite{A-M}, IV.4 and IV.7) one has
\[
H^*(G_{n-1} \wr \ZZ/p) = H^* (G_{n-1}^p, \: \ZZ/p)^{\ZZ/p} + \{ \Gamma (\alpha ) \cup \theta_i \} 
\]
where $H^* (G_{n-1}^p, \: \ZZ/p)^{\ZZ/p}$ is the image of the restriction map to $G_{n-1}^p$ and $\alpha \in H^k (G_{n-1}, \: \ZZ/p)$, $\Gamma (\alpha ) \in H^{kp} (G_{n-1} \wr \ZZ/p , \: \ZZ/p)$ is the total Steenrod power $\Gamma$ applied to $\alpha$, and $\theta_i$ are some classes in $H^i (\ZZ/p, \ZZ/p)$ with $i\ge 1$. A little more conceptionally, one can say that by Nakaoka's Theorem, as above, one has an isomorphism
\[
H^*(G_{n-1} \wr \ZZ/p, \: \ZZ/p) \simeq H^* (\ZZ/p, \: H^*(G_{n-1}, \: \ZZ/p)^{\otimes p})
\]
and there is a natural splitting
\begin{gather*}
H^* (G_{n-1} \wr \ZZ/p, \: \ZZ/p) = H^0 (\ZZ/p, \: H) \oplus \bigoplus_{k>0,\: T_{x}} H^k (\ZZ/p, \: T_{x})\\
\simeq H^* (G_{n-1}^p, \: \ZZ/p)^{\ZZ/p}\oplus \bigoplus_{k>0, \: T_{x}} H^k (\ZZ/p, \: T_{x})\, ,
\end{gather*}
the direct sum running over all the trivial modules $T_{x}$, generated by $x\otimes \dots \otimes x$, which occur. The point of the total Steenrod power $\Gamma$ is then that it realizes the generator of $H^k (\ZZ/p, \: T_{x})$ as an explicit class $\theta_i\cup \Gamma (x)$ in the cohomology of $G_n$ (here $\theta_i \in H^k (\ZZ/p, \: \ZZ/p)$) in a way that is compatible with cup products and functorial for group homomorphisms. We still need to remark that $\Gamma (x )$ restricts to $H^* (G_{n-1}^p, \: \ZZ/p)$ as $x\otimes \dots \otimes x$ and its restriction to $G_{n-1}\times \ZZ/p$ is of the form
\[
\sum_j D^j (x ) \cup g_j
\]
where, keeping in mind $H^* (G_{n-1}\times \ZZ/p, \: \ZZ/p) = H^* (G_{n-1}, \: \ZZ/p) \otimes H^* (\ZZ/p, \: \ZZ/p)$, the element $g_j$ is a generator of $H^j (\ZZ/p, \: \ZZ/p )$ and $D^j (x)$ is a certain class in $H^{\mathrm{deg}(x)p-j} (G_{n-1}, \: \ZZ/p)$. Here $D^j (x)$ is, if nontrivial, equal -up to a sign- to a Steenrod power $P^s (x)$ or $\beta P^s (x)$, where $\beta$ is the Bockstein operator. We do not need the exact formula (which is in \cite{A-M}, p. 184, (1.12)), but just that, due to the fact that the Steenrod operations $P^s$ and the Bockstein are zero in stable cohomology, $D^j (x)\cup g_j$ can just be nontrivial in stable cohomology if $D^j$ is the identity up to a sign and $j\le 1$. Which means $\mathrm{deg}(x)p-j = \mathrm{deg}(x)$, and this can happen for odd primes $p$ only when $x$ is in degree $0$. Hence we see that elements of the form $\Gamma (\alpha ) \cap \theta_i$ give nontrivial classes in stable cohomology (in view of Corollary \ref{cWreathElAb}, which of course we use all the time now) if and only if  this is in fact just a class $\theta_i$ coming from $H^1 (\ZZ/p, \: \ZZ/p )$.\\
Moreover the elements in $H^* (G_{n-1}^p, \: \ZZ/p)^{\ZZ/p}$ we are getting are exactly of two types: norms and traces; norms being elements of the form $x\otimes \dots \otimes x$, $x\in H^* (G_{n-1}, \: \ZZ/p)$ (these are stable if and only if $x$ is stable) and traces being of the form
\[
\sum_{i=0}^p \sigma^i (x_1 \otimes \dots \otimes x_p)
\]
$x_1\otimes \dots \otimes x_p\in H^* (G_{n-1}^p, \: \ZZ/p)$ and $\sigma$ is a cyclic shift. These are stable if and only if all the $x_i$ are stable.
\end{proof}

\section{Applications to finite groups of Lie type}\xlabel{sLie}

We would like to give some simple applications of the fore-going material to the computation of stable cohomology of some finite groups of Lie type. Consider the group $\mathrm{GL}_n (\mathbb{F}_q)$ of automorphisms of $\mathbb{F}_q^n$ where $\mathbb{F}_q$ is the finite field with $q$ elements.

\

We assume that $p$ and $q$ are coprime and odd. Then by \cite{A-M}, Theorem VII.4, 4.1 and Corollary 4.3, the $p$-Sylow subgroup $\mathrm{Syl}_p (\mathrm{GL}_n (\mathbb{F}_q)$ is a product of groups of the form: 
\[
\ZZ/p^{r} \wr \ZZ/p \dots \wr \ZZ/p \, .
\]
Hence Theorem \ref{tMain} shows

\begin{theorem}\xlabel{tFiniteGL}
Let $p$ and $q$ two odd primes with $(p, \: q) =1$, and let $n$ be an integer. Then the stable cohomology of $H^*_{\mathrm{s}}(\mathrm{GL}_n (\mathbb{F}_q), \: \ZZ/p)$ is detected by abelian subgroups (over any base field $k$). In particular, the unramified cohomology $H^*_{\mathrm{nr}}(\mathrm{GL}_n (\mathbb{F}_q), \: \ZZ/p)$ is trivial. The same statements hold for $\mathrm{SL}_n (\mathbb{F}_q)$ (the $p$-Sylow is the same as that of $\mathrm{GL}_n (\mathbb{F}_q)$).
\end{theorem}

One should compare this to the results obtained in \cite{BPT}. We hope to give a more exhaustive treatment of other classes of finite groups of Lie type elsewhere.

\end{document}